\documentclass{amsart}
\usepackage[utf8]{inputenc}
\usepackage{amsmath}
\usepackage{amsthm}
\usepackage{amsfonts, dsfont}
\usepackage{amssymb}
\usepackage[all]{xy}
\usepackage{indentfirst}
\usepackage[pagebackref=true]{hyperref}

\usepackage{cleveref}
\usepackage[alphabetic, initials]{amsrefs}

\usepackage{xfrac}

\newtheorem{thm}{Theorem}[section]

\newtheorem{theorem}[thm]{Theorem}
\newtheorem{corollary}[thm]{Corollary}
\newtheorem{proposition}[thm]{Proposition}

\newtheorem{lemma}[thm]{Lemma}
\newtheorem*{theorem*}{Theorem}
\newtheorem*{conjecture*}{Conjecture}
\newtheorem*{corollary*}{Corollary}

\theoremstyle{definition}

\newtheorem*{defn*}{Definiton}
\newtheorem{example}[thm]{Example}

\newtheorem*{ack}{Acknowledgements}
\newtheorem{remark}[thm]{Remark}

\newtheorem{question}[thm]{Question}

\newcommand{\Z}{\mathbb{Z}} %% Integers
\newcommand{\C}{\mathbb{C}} %% Complex
\newcommand{\T}{\mathbb{T}} %% Circle
\newcommand{\R}{\mathbb{R}}

\usepackage{verbatim} %% Enables 'comment'
\usepackage{color}
\newcommand{\G}{\Gamma}
\newcommand{\g}{\gamma}

\DeclareMathOperator{\Ima}{im}
\DeclareMathOperator{\interior}{int}

\newcommand{\La}{\Lambda}

\newcommand{\Sub}{\operatorname{Sub}}

\newcommand{\Prim}{\operatorname{Prim}}
\newcommand{\act}{\!\curvearrowright\!}

\newcommand{\Fix}{\operatorname{Fix}}

\newcommand{\cC}{\mathcal{C}}

\newcommand{\cI}{\mathcal{I}}

\newcommand{\cO}{\mathcal{O}}
\newcommand{\cP}{\mathcal{P}}

\newcommand{\id}{\mathrm{id}}

\renewcommand{\Prim}{\mathrm{Prim}}

\newcommand{\acts}{\curvearrowright}

\title{A tracial characterization of Furstenberg's $\times p,\times q$ conjecture}

\author{Chris Bruce}
\address[Chris Bruce]{School of Mathematics and Statistics, University of Glasgow, University Place, Glasgow G12 8QQ, United Kingdom}
\email{Chris.Bruce@glasgow.ac.uk}

\author{Eduardo Scarparo}
\address[Eduardo Scarparo]{Center for Engineering, Federal University of Pelotas\\ Brazil}
\email{eduardo.scarparo@ufpel.edu.br}
\thanks{This project has received funding from the European Research Council (ERC) under the European Union's Horizon 2020 research and innovation programme (grant agreement No. 817597). C. Bruce has received funding from the European Union’s Horizon 2020 research and innovation programme under the Marie Sklodowska-Curie grant agreement No 101022531.}

\subjclass[2020]{37A55, 46L05}

\begin{document}

%%%%%%%%%%%%%%%%
\maketitle
\begin{abstract}
We investigate almost minimal actions of abelian groups and their crossed products. As an application, given multiplicatively independent integers $p$ and $q$, we show that Furstenberg's $\times p,\times q$ conjecture holds if and only if the canonical trace is the only faithful extreme tracial state on the $C^*$-algebra of the group $\Z[\frac{1}{pq}]\rtimes\Z^2$.  We also compute the primitive ideal space and K-theory of $C^*(\Z[\frac{1}{pq}]\rtimes\Z^2)$.
\end{abstract}

\section{Introduction}
Let $p,q\geq 2$ be two \emph{multiplicatively independent} integers, in the sense that there are no $r,s\in\Z_{>0}$ such that $p^r=q^s$. In \cite{F67}, Furstenberg showed that the only infinite closed $\times p,\times q$-invariant subset of $\frac{\R}{\Z}\simeq\T$ is the whole circle. The measure-theoretic analogue of this result is now one of the most fundamental open problems in ergodic theory (for a survey on this problem, see, for example, \cite{Linden}). It can be formulated precisely as follows:

\begin{conjecture*}[Furstenberg's $\times p,\times q$ conjecture]
The only ergodic $\times p,\times q$-invariant probability measure on $\T$ with infinite (hence full) support is the Lebesgue measure.
\end{conjecture*}

There have already been two connections between the above conjecture and $C^*$-algebra theory. Following an idea from Cuntz, Huang and Wu in \cite{HW17} gave a characterization of Furstenberg's conjecture in terms of irreducible representations of the group $\Z[\frac{1}{pq}]\rtimes\Z^2$. 
In \cite{S-T20}, Furstenberg's theorem on closed $\times p,\times q$-invariant subsets of $\T$ was used to show that every non-zero ideal $I\unlhd C^*(\Z[\frac{1}{pq}]\rtimes\Z^2)$ intersects $\C[\Z[\frac{1}{pq}]\rtimes\Z^2]$ non-trivially. Since there is a non-trivial, but well-understood, relationship between $C^*$-simplicity of a group (that is, simplicity of the reduced group $C^*$-algebra) and uniqueness of the canonical tracial state on the reduced group $C^*$-algebra (\cite{BKKO}), it is, in light of \cite{S-T20}, natural to ask if Furstenberg's $\times p,\times q$ conjecture has a $C^*$-algebraic manifestation in terms of tracial states on $C^*(\Z[\frac{1}{pq}]\rtimes\Z^2)$.
In this work, we show that this is indeed the case: We present a characterization of the conjecture in terms of the set of tracial states on $C^*(\Z[\frac{1}{pq}]\rtimes\Z^2)$, see Corollary~\ref{thm:tf}.

Since $C^*(\Z[\frac{1}{pq}]\rtimes\Z^2)$ can be realized as a crossed product $C^*$-algebra for a certain almost minimal action of $\Z^2$, we begin with several general results on tracial states and primitive ideals for crossed products arising from such actions. Recall that an action of a group $\G$ on a compact space $X$ is said to be \emph{almost minimal} if every invariant closed set $F\subsetneq X$ is finite.  Almost minimality is also known in the literature as \emph{irreducibility} or as the \emph{ID property} (see, for example, \cite{S95}*{Section 29} for an account of this property).

In Section \ref{sec:am}, we compute the primitive ideal space and the set of extreme tracial states of the crossed product associated to an almost minimal action of an abelian group. Along the way, we fix a mistake in the literature concerning the computation of the primitive ideal spaces of a special class of such crossed products (see Remark \ref{rem:fix})

In Section \ref{ref:sr}, we apply the results of the previous section for computing the primitive ideal space and the set of extreme tracial states on $C^*(\Z[\frac{1}{pq}]\rtimes\Z^2)$. As an application, we show that Furstenberg's $\times p,\times q$ conjecture holds if and only if the canonical trace is the only faithful extreme tracial state on $C^*(\Z[\frac{1}{pq}]\rtimes\Z^2)$. We also compute the K-theory of $C^*(\Z[\frac{1}{pq}]\rtimes\Z^2)$.

Finally, in Section \ref{sec:q}, we pose a few questions that arise naturally in light of our results.
\begin{ack}
We thank the referee for several useful suggestions and for pointing out a gap in the original argument of Theorem~\ref{thm:kt}.
\end{ack}

\section{Almost minimal actions}\label{sec:am}

\subsection{Preliminaries}

Throughout this paper, $\G\act X$ denotes an action of a group $\G$ on a compact Hausdorff space $X$ by homeomorphisms. Given $x\in X$, we let $\G_x:=\{g\in \G:gx=x\}$ be the \emph{stabilizer} of $x$, and $[x]:=\{gx:g\in\G\}$ be the \emph{orbit} of $x$.

Given $g\in\G$, let $\Fix_g:=\{x\in X:gx=x\}$. We say that the action is \emph{faithful} if $\Fix_g\subsetneq X$ for any $g\in\G\setminus\{e\}$. The action is said to be \emph{topologically free} if, for every $g\in\G\setminus\{e\}$, $\interior\Fix_g=\emptyset$. If $\G$ is countable, then topological freeness is equivalent to the set $\{x\in X:\G_x=\{e\}\}$ being dense in $X$ by a Baire category argument. 

Let $\cP_\G(X)$ be the space of $\G$-invariant regular probability measures on $X$. We say that $\mu\in\cP_\G(X)$ is \emph{essentially free} if, for any $g\in \G\setminus\{e\}$, $\mu(\Fix_g)=0$. If $\G$ is countable, then $\mu$ is essentially free if and only if $\mu(\{x\in X:\G_x=\{e\}\})=1$.

The action is said to be \emph{almost minimal} if every invariant closed set $F\subsetneq X$ is finite. If $\G\act X$ is almost minimal, then any infinite orbit is dense in $X$. 

\begin{example}
    Let $\alpha$ be an action by homeomorphisms on a non-compact, locally compact Hausdorff space $X$. If $\alpha$ is \emph{minimal} (that is, every orbit is dense), then the extension of $\alpha$ to the one-point compactification of $X$ is almost minimal.
\end{example}

Certain almost minimal algebraic actions were studied by Berend (\cite{Ber83,Ber84}) and by Laca and Warren (\cite{LW20}). We also refer the reader to Schmidt's book \cite{S95}*{Section 29} for further discussions of examples from algebraic actions.

Any action on a finite space is almost minimal, but if $\G\act X$ is an almost minimal action and $X$ is infinite, then $\G$ is infinite:
\begin{proposition}\label{prop:inf}
    Let $X$ be an infinite compact space and $\G\act X$ an almost minimal action. Then $\G$ is infinite.
\end{proposition}
\begin{proof}
    Suppose $\G$ is finite. Take $x,y\in X$ such that $[x]\cap[y]=\emptyset$. We will show that there are disjoint $\G$-invariant open neighbourhoods $W_x$ and $W_y$ of $[x]$ and $[y]$, respectively. This suffices to prove the result by the following argument: Since $W_x^c\cup W_y^ c= X$, we have that either $W_x^c$ or $W_y^c$ is infinite, hence equal to $X$, thus contradicting that $W_x$ and $W_y$ are not empty.

Now let us construct such $W_x$ and $W_y$. Given an open neighbourhood $U$ of $x$, the set $V:=\bigcap_{g\in\G_x}gU$ is a $\G_x$-invariant open neighbourhood of $x$. Let $g_1,\dots,g_n$ be left coset representatives for $\G/\G_x$. By taking $U$ smaller, we may assume that $g_1 V,\dots,g_n V$ are disjoint sets. Then $W_x:=\bigcup_{i=1}^n g_i V$ is a $\G$-invariant open neighbourhood of $[x]$. Moreover, by taking $U$ smaller, we can assume that $\overline{W_x}\cap [y]=\emptyset$. By constructing a neighbourhood $W_y$ of $[y]$ in an analogous way, such that $W_y\subseteq \overline{W_x}^c$, we obtain the open sets with the desired properties.
\end{proof}
For the proof of the next result, notice that, if $h$ is a homeomorphism on a space $X$, then, given $Y\subset X$, we have that $h((\interior Y)^c)=(h(\interior Y))^c=(\interior h(Y))^c$. In particular, if $h(Y)=Y$, then $h((\interior Y)^c)=(\interior Y)^c$ as well.
\begin{proposition}\label{prop:tf}
Let $\G$ be a countable abelian group, $X$ an infinite compact space and $\G\act X$ a faithful almost minimal action. Then the action $\G\act X$ is topologically free, and the set of points that have finite orbits is countable and has empty interior.

\end{proposition}
\begin{proof}
Given $g\in\G\setminus\{e\}$, we have that $\Fix_g\subsetneq X$ is closed and invariant, hence finite.  Since $(\interior\Fix_g)^c$ is closed, infinite, and invariant, it follows that $(\interior\Fix_g)^c=X$, hence $\interior\Fix_g=\emptyset$. Therefore, the action is topologically free.

Suppose $x\in X$ has finite orbit. Since $\G$ is infinite by Proposition \ref{prop:inf}, we have that $x\in\Fix_g$ for some $g\in\G\setminus\{e\}$. Since $\Fix_g$ is finite and has empty interior for any $g\in\G\setminus\{e\}$, the result follows from the Baire category theorem.
\end{proof}

\subsection{Primitive ideals}
Recall that an ideal $I$ of a $C^*$-algebra $A$ is said to be \emph{primitive} if there exists a non-zero irreducible representation $\varphi\colon A \to B(H)$ such that $I=\ker(\varphi)$. The \emph{primitive ideal space of $A$}, denoted by $\Prim(A)$, is the set of primitive ideals endowed with the \emph{hull-kernel topology} (see, for instance, \cite[Section 5.4]{M90}).

In the proof of the next result, we use a description from \cite{W81} of the primitive ideal space of the crossed product associated to an action of an abelian group. We also use some ideas from \cite[Section~5]{LW20} (beware that there is an issue with the description of convergent nets in \cite[Theorem~5.2]{LW20}, see Remark~\ref{rem:fix}). For an alternative approach to describing the primitive ideal space, see the proof of \cite[Theorem~9.D.1]{BdlH20}.

\begin{theorem}\label{thm:pis}
Let $\G$ be a countable abelian group, $X$ an infinite, second countable compact space, and $\G\act X$ a faithful and almost minimal action. Then the set $\mathcal{C}$
of finite orbits is countable, and $\Prim(C(X)\rtimes\G)$ is homeomorphic to
$$P:=\left(\bigsqcup_{[x]\in \mathcal{C}}\{[x]\}\times\widehat{\G_x}\right)\sqcup\{\infty\},$$
where the topology on $P$ is defined as follows: the closed subsets of $P$ are $P$ and finite unions of elements from the collection $\{\{[x]\}\times F:\text{$[x]\in\mathcal{C}$, $F$ a closed subset of $\widehat{\G_x}$}\}$.
\end{theorem}
\begin{proof}
By Proposition \ref{prop:tf}, $\G\act X$ is topologically free, and $\cC$ is countable.

Consider the equivalence relation on $X\times\widehat{\G}$ defined by $(x,\eta)\sim(y,\chi)$ if $\overline{[x]}=\overline{[y]}$ and $\eta|_{\G_x}=\chi|_{\G_x}$ (note that, since $\G$ is abelian, $\overline{[x]}=\overline{[y]}$ implies that $\G_x=\G_y$). By \cite[Theorem 5.3]{W81}, $\Prim(C(X)\rtimes\G)$ is homeomorphic to the quotient space $\frac{X\times\widehat{\G}}{\sim}$.

Given $x\in X$ such that $[x]\notin\cC$, it follows from almost minimality and topological freeness of the action that $\overline{[x]}=X$ and $\G_x=\{e\}$. Let $f\colon \frac{X\times\widehat{\G}}{\sim}\to P$ be given by $f([x,\chi]):=([x],\chi|_{\G_x})$ if $[x]\in\mathcal{C}$, and $f([x,\chi]):=\infty$ otherwise. Given $x\in X$, let $r_x\colon\widehat{\G}\to\widehat{\G_x}$ be the restriction map. It is not difficult to see that $f$ is bijective (surjectivity follows from surjectivity of each $r_x$).

Let $\pi\colon X\times \widehat{\G}\to \frac{X\times\widehat{\G}}{\sim}$ be the quotient map. We will show that $f$ is continuous. For this, it suffices to show that $f\circ\pi$ is continuous. Given $[x]\in\cC$ and $F\subseteq\widehat{\G_x}$ closed, we have that $(f\circ\pi)^{-1}(\{[x]\}\times F)=[x]\times r_x^{-1}(F)$ is closed in $X\times\widehat{\G}$. Therefore, $f\circ\pi$ is continuous.

Let us now show that $f^{-1}$ is continuous, that is, that $f$ is a closed map. Let $A\subseteq \frac{X\times\widehat{\G}}{\sim}$ be a closed subset. We will show that $f(A)$ is closed in $P$. 

Case $\infty\in f(A)$: In this case, $\{x\in X:[x]\notin\cC\}\times\widehat{\G}\subseteq \pi^{-1}(A)$. Since the set of points with infinite orbit is dense in $X$ by Proposition \ref{prop:tf}, it follows that $\pi^{-1}(A)=X\times\widehat{\G}$, hence $f(A)=f(\pi(\pi^{-1}(A)))=P$.

Case $I_A:=\{[x]\in\cC:f(A)\cap(\{[x]\}\times\widehat{\G_x})\neq\emptyset\}$ is infinite: Since $U:=\bigcup I_A$ is $\G$-invariant and infinite, it follows from almost minimality that U is dense in $X$. This implies that $\{x : \exists\chi\in\widehat{\G} \text{ with } (x,\chi)\in\pi^{-1}(A)\}=X$. Hence, there exists $(x,\chi)\in\pi^{-1}(A)$ such that $[x]\notin\cC$. In particular, $\infty\in f(A)$ and $f(A)=P$ by the previous case.

Finally, assume that $I_A$ is finite and $\infty\notin f(A)$. In this case, there exists a family $(B_{[x]})_{[x]\in I_A}$ such that each $B_{[x]}$ is a subset of $\widehat{\G_x}$, and $f(A)=\bigsqcup_{[x]\in I_A}\{[x]\}\times B_{[x]}$. In order to conclude that $f(A)$ is closed, we have to show that each $B_{[x]}$ is closed in $\widehat{\G_x}$. We have that $\pi^{-1}(A)=\bigsqcup_{[x]\in I_A}[x]\times r_x^{-1}(B_{[x]})$. Given $[x]\in I_A$, since $\pi^{-1}(A)$ is closed in $X\times\widehat{\G}$, it follows that $r_x^{-1}(B_{[x]})$ is closed in $\widehat{\G}$. Since $r_x$ is a quotient map (being a surjective continuous map between compact spaces), we conclude that $B_{[x]}$ is closed in $\widehat{\G_x}$.
\end{proof}

\begin{remark}\label{rem:fix}
Given $\G\act X$ and $P$ as in Theorem \ref{thm:pis}, we have that $\infty$ is a dense point in $P$ (this corresponds to the fact that the ideal $\{0\}$ is a dense point in $\Prim(C(X)\rtimes\G)$). Moreover, given a net $([x_i],\chi_i)$ in $P\setminus\{\infty\}$, we have
\begin{enumerate}
    \item[(a)] if for every finite set $F\subseteq \cC$, the net $[x_i]$ is eventually outside $F$, then $([x_i],\chi_i)$ converges to every point in $P$;
    \item[(b)] if $[x_i]$ is eventually constant, $[x_i]=[x]$ for all $i\geq j$ say, then $([x_i],\chi_i)$ converges to $([y],\chi)$ if and only if $[y]=[x]$ and $\chi=\lim_{i\geq j}\chi_i$.
\end{enumerate}
The paper \cite{LW20} considers a special class of algebraic actions coming from algebraic number theory. Theorem~\ref{thm:pis} corrects an error in \cite[Section~5]{LW20}: Convergent nets as in (a) above are not accounted for in the description of the topology on the primitive ideal space stated in \cite[Theorem~5.2]{LW20}. Note that the class of actions considered in \cite[Theorem~5.2]{LW20} satisfy the hypothesis of Theorem~\ref{thm:pis} by \cite[Theorem~4.3]{LW20}.
\end{remark}

\subsection{Tracial states}
Given a group $\G$, let $\Sub(\G)$ be the set of subgroups of $\G$, endowed with the \emph{Chabauty topology}; this is the restriction to $\Sub(\G)$ of the product topology on $\{0,1\}^\G$, where every $\La\in\Sub(\G)$ is identified with its characteristic function $1_\La\in\{0,1\}^\G$. We also endow $\Sub(\G)$ with the action of $\G$ given by conjugation.

Given a convex set $K$, let $\partial K$ be the set of extreme points of $K$.

\begin{proposition}\label{prop:ef}
Let $\G$ be a countable abelian group and $\G\act X$ a topologically free action on a compact space $X$. Given $\mu\in\partial\cP_\G(X)$ with full support, we have that $\mu$ is essentially free.
\end{proposition}
\begin{proof}
Suppose that $\mu(\{x\in X:\G_x\neq\{e\}\})>0$. The map 
\begin{align*}
S\colon X&\to\Sub(\G)\\
x&\mapsto \G_x
\end{align*}
is $\G$-invariant and Borel measurable. Since $\G$ is abelian, the action $\G\act\Sub(\G)$ is trivial. Given any Borel measurable set $U\subseteq \Sub(\G)$, by ergodicity of $\mu$ we have that $S_*\mu(U)=\mu(S^{-1}(U))\in\{0,1\}$. Thus, there exists $\{e\}\lneq\La\leq\G$ such that $S_*\mu=\delta_\Lambda$, the point-mass measure concentrated at $\Lambda$. It follows that $\mu(\{x\in X:\G_x=\La\})=1$. Since $\mu$ has full support, we have that $\{x\in X:\G_x=\La\}$ is dense. This implies that, for any $x\in X$, it holds that $\La\leq\G_x$. But by topological freeness, there exists $x\in X$ such that $\G_{x}=\{e\}$, which gives a contradiction.
\end{proof}

Given a $C^*$-algebra $A$, we denote by $T(A)$ the set of tracial states on $A$. We say that $\tau\in T(A)$ is \emph{faithful} if, for any $a\in A\setminus\{0\}$, we have that $\tau(a^*a)>0$.

Let $\G$ be an abelian group acting on a compact space $X$. Given $x\in X$ with finite orbit and $\chi\in\widehat{\G_x}$, let $\tau_{x,\chi}$ be the state on $C(X)\rtimes\G$ such that, given $f\in C(X)$ and $g\in\G$,
\begin{align}\label{eq:ft}
\tau_{x,\chi}(fu_g)=
\begin{cases}
\frac{\chi(g)}{|\G x|}\sum_{y\in[x]} f(y) &\text{ if $g\in\G_x$},\\
0 &\text{ otherwise.} 
\end{cases}
\end{align}
The fact that $\tau_{x,\chi}$ is well-defined follows, for example, from \cite[Corollary 2.5.12 and Exercise 4.1.4]{BO08}. Here, and throughout this paper, we use $u_g$ to denote both the canonical unitary in the group C*-algebra $C_r^*(\G)$ or the canonical unitary in the crossed product $C(X)\rtimes_r\G$ corresponding to an element $g\in\G$.

Given $\mu\in\cP_\G(X)$, let $\tau_\mu$ be the state on $C(X)\rtimes\G$ such that, given $f\in C(X)$ and $g\in\G$,
\begin{align}\label{eq:at}
\tau_\mu(fu_g)=
\begin{cases}
\int_X f\,d\mu &\text{if $g=e$},\\
0 &\text{otherwise.} 
\end{cases}
\end{align}

The following result is an immediate consequence of Propositions \ref{prop:tf} and \ref{prop:ef}, together with \cite[Corollary 2.4]{N13} (for an alternative approach to Neshveyev's result, see \cite[Theorem 2.7]{KTT} and the proof of \cite[Theorem 12.D.1]{BdlH20}).

\begin{proposition}\label{prop:taf}
Let $\G$ be a countable abelian group, $X$ an infinite, second countable compact space, and $\G\act X$ a faithful and almost minimal action. Then the set $\mathcal{C}$
of finite orbits is countable, and there is a bijective correspondence between 
$$\left(\bigsqcup_{ [x]\in \mathcal{C}}\{[x]\}\times\widehat{\G_x}\right)\sqcup\{\mu\in\partial\cP_\G(X):\mu\text{ has full support}\}$$
and $\partial T(C(X)\rtimes \G)$, which maps $\mu\in\partial\cP_\G(X)$ with full support to $\tau_\mu$ as in \eqref{eq:at}, and, given $[x]\in\mathcal{C}$, maps $([x],\chi)\in\{[x]\}\times\widehat{\G_x}$ to $\tau_{x,\chi}$ as in \eqref{eq:ft}.
\end{proposition}

\section{Some observations on $C^*(\Z[1/pq]\rtimes\Z^2)$}\label{ref:sr}

\subsection*{Preliminaries}
\label{ss:prelims}
Fix integers $p,q\geq 2$ and let $\alpha$ be the action of $\Z^2$ on $\Z[\frac{1}{pq}]$ given by $\alpha_{(n,m)}(x):=p^nq^mx$, for $n,m\in\Z$, and $x\in \Z[\frac{1}{pq}]$. Also let
\begin{equation*}
\varphi\colon \mathbb{T}\to \mathbb{T}, \quad z\mapsto z^{pq},
\end{equation*}
and consider the compact abelian group
\begin{equation*}
X:=\varprojlim(\mathbb{T},\varphi)=\{(x_n)_{n\in\Z_{\geq 0}}\in\prod_{n\in\Z_{\geq 0}}\mathbb{T}:\text{ for all } n\in\Z_{\geq0},x_n=\varphi(x_{n+1})\}.
\end{equation*}
There is a homeomorphism (which is also a group isomorphism) $H\colon\widehat{\Z[\frac{1}{pq}]}\to X$ given by 
\begin{equation}\label{eq:h}
H(\tau):=(\tau(\frac{1}{(pq)^n}))_{n\in\Z_{\geq 0}},
\end{equation}
for $\tau\in\widehat{\Z[\frac{1}{pq}]}$ (see, for example, the proof of \cite[Lemma 2.3]{S-T20}). 

Let $S\colon X\to X$ be the left shift map and $T_p,T_q\colon X\to X$ be the maps given by $T_p(x):=x^p$ and $T_q(x):=x^q$, for $x\in X$. Then $T_p$ and $T_q$ are the continuous group automorphisms of $X$ associated with the multiplication-by-$p$ and the multiplication-by-$q$ automorphisms of $\Z[\frac{1}{pq}]$ via the homeomorphism $H$. Moreover, $T_p$ and $T_q$ satisfy $T_p^{-1}=ST_q$ and $T_q^{-1}=ST_p$. Let $\beta:\Z^2\act X$ be given by 
\begin{equation}\label{eq:beta}
\beta_{(r,s)}=T_p^{-r}T_q^{-s},
\end{equation}
for $(r,s)\in\Z^2$. In particular, $\beta_{(1,1)}=S$. 

Furthermore, one can easily check that $H$ cojugates $\widehat{\alpha}\colon\Z^2\act \widehat{\Z[\frac{1}{pq}]}$ with $\beta$, where $\widehat{\alpha}\colon\Z^2\acts\widehat{\Z[\frac{1}{pq}]}$ is the action $\widehat{\alpha}_{(r,s)}(\chi):=\chi\circ\alpha_{(-r,-s)}$ for all $(r,s)\in\Z^2$ and $\chi\in \widehat{\Z[\frac{1}{pq}]}$. Hence $C^*(\Z[\frac{1}{pq}]\rtimes\Z^2)\simeq C(X)\rtimes\Z^2$.

 Let $\varphi_p,\varphi_q\colon\T\to\T$ be given by $\varphi_p(z):=z^p$ and $\varphi_q(z):=z^q$, for $z\in \T$, and
$$
\cI:=\{B\subseteq\T:\varphi_p(B)=\varphi_q(B)=B\}.
$$

Given $i\in\Z_{\geq 0}$, let $\pi_i\colon X\to \T$ be defined by $\pi_i((x_n)):=x_i$, for $(x_n)\in X$. 

\begin{proposition}\label{prop:inv} There is a bijection between closed $\Z^2$-invariant subsets $F\subseteq X$ and closed sets $B\in\cI$, which maps $F$ into $\pi_0(F)$. Moreover, a $\Z^2$-invariant subset $F\subset X$ is finite if and only if $\pi_0(F)$ is finite. 
\end{proposition}
\begin{proof}
    Given $F\subset X$ closed and $\Z^2$-invariant, it follows from the definition of the $\Z^2$-action in \eqref{eq:beta} that $\pi_0(F)\in\cI$. Since $F$ is shift-invariant (recall that $S=T_p^{-1}T_q^{-1}$), we also have that $\pi_0(F)=\pi_n(F)$ for every $n\geq 0$. 
    
    We claim that $F=\left(\prod_{n\geq 0}\pi_n(F)\right)\cap X$. Clearly, $F\subset\left(\prod_{n\geq 0}\pi_n(F)\right)\cap X$. Conversely, given $x\in \left(\prod_{n\geq 0}\pi_n(F)\right)\cap X$, for each $m\in \Z_{\geq 0}$ there is $y\in F$ such that $\pi_n(x)=\pi_n(y)$ for $n\leq m$. Since $F$ is closed, we conclude that $x\in F$, thus showing the claim. In particular, $F=\left(\prod_{n\geq 0}\pi_0(F)\right)\cap X$.

    Given $B\in \cI$, we have that $F:=\left(\prod_{n\geq 0}B\right)\cap X$ is closed, $\Z^2$-invariant and $\pi_0(F)=B$. This concludes the proof of the first claim.

    Suppose $B:=\pi_0(F)$ is finite. Given $z\in B$, since $B=B^{pq}$ and $B$ is finite, there is a unique $w\in B$ such that $w^{pq}=z$. This shows that any $x\in F$ is uniquely determined by $\pi_0(x)$. Therefore, $|B|=|F|$. 
\end{proof}

Recall that integers $p,q\geq 2$ are said to be multiplicatively independent if $p^r\neq q^s$ for all $r,s\in\Z_{>0}$. Furstenberg's theorem (\cite[Part IV]{F67}) and Proposition \ref{prop:inv} imply that $\Z^2\act X$ is almost minimal. By \cite[Lemma 2.2]{S-T20} (or Proposition \ref{prop:tf}), $\Z^2\act X$ is topologically free. For future reference, let us record these facts in the following result: 

\begin{lemma}
\label{lem:xpxq}
Let $p,q\geq 2$ be multiplicatively independent integers. Then $\Z^2\act X$ is almost minimal and topologically free.
\end{lemma}

Let $\cO$ be the set of finite minimal (among the non-empty sets) elements of $\cI$. Notice that $\cO$ is in one-to-one correspondence with the set of finite orbits of $\Z^2\act X$. For any $r\in\Z_{>0}$ coprime with $p$ and $q$, we have that $\{e^{\frac{2\pi i}{r}}:1\leq i<r\}\in\cI$. In particular, $\cO$ is infinite. The following is an immediate consequence of Lemma~\ref{lem:xpxq}, Theorem~\ref{thm:pis}, and the isomorphism $C^*(\Z[\frac{1}{pq}]\rtimes\Z^2)\simeq C(X)\rtimes\Z^2$. Note that if a point of $X$ has finite $\Z^2$-orbit, then its stabilizer subgroup is of finite index in $\Z^2$ and is thus isomorphic to $\Z^2$.

\begin{theorem}
Let $p ,q\geq 2$ be multiplicatively independent integers. Then the primitive ideal space of $C^*(\Z[\frac{1}{pq}]\rtimes\Z^2)$ is homeomorphic to
$$P:=\left(\cO\times\T^2\right)\sqcup\{\infty\},$$
where the closed subsets of $P$ are $P$ and finite unions of elements of the collection $\{\{B\}\times F:\text{$B\in\mathcal{O}$, $F$ a closed subset of $\T^2$}\}$.
\end{theorem}

Let $\cP_{p,q}(\T)$ be the set of regular probability measures $\mu$ on $\T$ such that $(\varphi_p)_*\mu=(\varphi_q)_*\mu=\mu$. Since $\varphi_{p}\circ\pi_0=\pi_0\circ T_p$ and $\varphi_q\circ\pi_0=\pi_0\circ T_q$, it follows that $(\pi_0)_*\colon \cP_{\Z^2}(X)\to \cP_{p,q}(\T)$ is a well-defined affine map. The following result is known (see \cite[Proposition 4.1]{HW17}), but for the reader's convenience we provide a different proof.
\begin{lemma}\label{lem:ob}
The map $(\pi_0)_*\colon \cP_{\Z^2}(X)\to \cP_{p,q}(\T)$ is an affine isomorphism.
\end{lemma}
\begin{proof}
 We will show that $(\pi_0)_*$ is bijective by constructing an inverse. Given $n\in\Z_{\geq 0}$, let $\pi_n^*\colon C(\T)\to C(X)$ be the adjoint map given by $\pi_n^*(f):=f\circ\pi_n$ for $f\in C(\T)$, and $A_n:=\pi_n^*(C(\T))$. Notice that $\pi_n^*$ is is injective. For every $n\in\Z_{\geq 0}$, we have that 
\begin{equation}\label{eq:inv}\pi_n^*=\pi_{n+1}^*\circ\varphi^*.
\end{equation}
In particular, $A_n\subseteq A_{n+1}$.

Fix $\mu\in\cP_{p,q}(\T)$. By \eqref{eq:inv} and invariance of $\mu$, there exists a bounded, positive and unital linear functional $\sigma$ on $\bigcup_{n\in\Z_{\geq 0}}A_n$ such that, for $f\in C(\T)$ and $n\in\Z_{\geq 0}$, we have that $\sigma(\pi_n^*(f))=\int f\,d\mu$. Since $C(X)=\overline{\bigcup_{n\in\Z_{\geq 0}}A_n}$, we can extend $\sigma$ to a state on $C(X)$. Let $\psi(\mu)\in \cP(X)$ be the measure corresponding to $\sigma$. By $\times p,\times q$-invariance of $\mu$, we have that $\psi(\mu)$ is $\Z^2$-invariant. 

By $T_p,T_q$-invariance, we have that $(\pi_n)_*|_{\cP_{\Z^2}(X)}=(\pi_0)_*|_{\cP_{\Z^2}(X)}$ for any $n\geq 0$. This shows that any $\nu\in\cP_{\Z^2}$ is uniquely determined by its restriction to $A_0$. Using this fact, it is easy to check that $\psi\colon \cP_{p,q}(\T)\to \cP_{\Z^2}(X) $ is an inverse for $(\pi_0)_*|_{\cP_{\Z^2}(X)}$. 
\end{proof}

Let $\mathrm{Ev}\colon C^*(\Z[\frac{1}{pq}])\to C(\widehat{\Z[\frac{1}{pq}]})$ be the isomorphism given by point-evaluation, that is, given $g\in\Z[\frac{1}{pq}]$ and $\chi\in\widehat{\Z[\frac{1}{pq}]}$, we have that $\mathrm{Ev}(u_g)(\chi)=\chi(g)$. One can easily check that $\mathrm{Ev}$ conjugates the canonical actions $\Z^2\act C^*(\Z[\frac{1}{pq}])$ and $\Z^2\act C(\widehat{\Z[\frac{1}{pq}]})$.

\begin{proposition}\label{prop:tpq}
Let $p ,q\geq 2$ be multiplicatively independent integers. Then there is a bijection between $\partial T(C^*(\Z[\frac{1}{pq}]\rtimes\Z^2))$ and
\[\left(\cO\times\T^2\right)\sqcup\{\mu\in\partial\cP_{p,q}(\T):\mu\text{ has full support}\}\]
as in Proposition \ref{prop:taf}. This bijection takes the non-faithful tracial states into $\cO\times\T^2$, and takes each faithful tracial state $\tau$ to $\mu\in\partial\cP_{p,q}(\T)$ with full support such that, for $n\in\Z,$ 
\begin{equation}\label{eq:cha}
\tau(u_{(n,0,0)})=\int_\T z^n\,d\mu(z).
\end{equation}
\end{proposition}
\begin{proof}
Identify $C^*(\Z[\frac{1}{pq}]\rtimes\Z^2)$ with $C(X)\rtimes\Z^2$. Notice that the tracial states in \eqref{eq:ft} are not faithful, whereas, given $\mu\in\cP_{\Z^2}(X)$ with full support, the tracial states as in \eqref{eq:at} are faithful. Together with Proposition \ref{prop:taf}, this shows the first claims.

Let us prove \eqref{eq:cha}. Identify $C^*(\Z[\frac{1}{pq}]\rtimes\Z^2)$ with $C^*(\Z[\frac{1}{pq}])\rtimes\Z^2$. By Proposition \ref{prop:taf}, there exists $\nu\in\partial \cP_{\Z^2}(X)$ with full support such that 
\begin{equation}\label{eq:mu}
\tau\circ\mathrm{Ev}^{-1} (f\circ H)=\int_X f\,d\nu
\end{equation}
for any $f\in C(X)$. Let $\mu:=(\pi_0)_*(\nu)$. Given $n\in\Z$, we have that

\begin{equation}\label{eq:tec}
\int_\T z^n\,d\mu(z)=\int_X x_0^n d\nu(x).
\end{equation}
Given $\chi\in\widehat{\Z[\frac{1}{pq}]}$, we also have that $H(\chi)_0^n=\chi(n)=\mathrm{Ev}(u_n)(\chi)$ (here, $H$ is the isomorphism defined in \eqref{eq:h}). Together with \eqref{eq:mu} and \eqref{eq:tec}, this concludes the proof of \eqref{eq:cha}. 
\end{proof}

\subsection*{Tracial states}
A group $\G$ is said to be \emph{icc} if the conjugacy class of any $g\in\G\setminus\{e\}$ is infinite. Here, icc stands for \emph{infinite conjugacy classes}.

The \emph{canonical trace} on $C^*_r(\G)$ is the faithful tracial state $\tau$ on $C^*_r(\G)$ which satisfies $\tau(u_g)=0$ for all $g\in\G\setminus\{e\}$. Recall that $\G$ is icc if and only if the canonical trace is an extreme tracial state on $C^*_r(\G)$ (\cite[Propositions 7.A.1 and 11.C.3]{BdlH20}). 

\begin{lemma} Let $p,q\geq 2$ be multiplicatively independent integers. Then $\Z[\frac{1}{pq}]\rtimes\Z^2$ is icc. 
\end{lemma}
\begin{proof}
Given $x\in\Z[\frac{1}{pq}]\setminus\{0\}$, $y\in\Z^2$ and $n\in\Z$, we have that $$(0,n,0)(x,y)(0,-n,0)=(p^nx,y).$$ Therefore, the conjugacy class of $(x,y)$ is infinite.

Given $(m,n)\in\Z^2\setminus\{0,0\}$ and $x\in \Z[\frac{1}{pq}]$, we have that 
$$(x,0)(0,m,n)(-x,0)=((1-p^mq^n)x,m,n).$$
From multiplicative independence of $p$ and $q$, we conclude that the conjugacy class of $(0,m,n)$ is infinite as well.
\end{proof}

Since the canonical trace corresponds to the Lebesgue measure on $\T$ under the bijection of Proposition \ref{prop:tpq}, the following holds:

\begin{corollary}\label{thm:tf}
Let $p,q\geq 2$ be multiplicatively independent integers.  Then the canonical trace is the only faithful extreme tracial state on $C^*(\Z[\frac{1}{pq}]\rtimes \Z^2)$ if and only if Furstenberg's $\times p,\times q$ conjecture holds.
\end{corollary}

For the purpose of the following discussion, let us say that an icc group $\G$ has the \emph{weak unique trace property} if the canonical trace is the only faithful extreme tracial state on $C^*_r(\G)$ and is \emph{weakly $C^*$-simple} if every non-zero ideal $I\unlhd C^*_r(\G)$ intersects $\C\G$ non-trivially.  In \cite{S-T20}, it was shown that, for $p,q\geq 2$ multiplicatively independent integers, $\Z[\frac{1}{pq}]\rtimes\Z^2$ is weakly $C^*$-simple.  Furthermore,  Corollary \ref{thm:tf} can be rephrased as saying that Furstenberg's $\times p,\times q$ conjecture holds if and only if $\Z[\frac{1}{pq}]\rtimes\Z^2$ has the weak unique trace property.

Since $C^*$-simplicity implies the unique trace property by \cite[Corollary 4.3]{BKKO}, one could naively wonder whether weak $C^*$-simplicity implies the weak unique trace property (by the above, if this were true, it would imply Furstenberg's conjecture). Unfortunately, this does not hold in general, as the followiong example shows.

\begin{example}
\label{ex:Vaes}
It was observed by Ozawa that the \emph{lamplighter group} $\Z_2\wr\Z$ is weakly $C^*$-simple (\cite{Alekseevoberwolfach}). We claim that $\Z_2\wr\Z$ does not have the weak unique trace property.  Given $t\in(0,1)$, let $\mu_t:=\prod_\Z(t\delta_0+(1-t)\delta_1)\in\partial\cP_\Z(\{0,1\}^\Z)$. Since $\mu_t$ has full support, it follows from \cite[Corollary 2.4]{N13} that $\mu_t$ gives rise to a faithful extreme tracial state on $C^*(\Z_2\wr\Z)\simeq C(\{0,1\}^\Z)\rtimes\Z$ (this has been observed independently by Vaes in \cite{V22}).  
\end{example}

\subsection*{K-theory}
Given an automorphism $\alpha$ on a $C^*$-algebra $A$, let $\iota\colon A\to A\rtimes\Z$ be the canonical embedding. The \emph{Pimsner-Voiculescu sequence} is the following exact sequence (see, for instance, \cite{PV80}): 
\[
\xymatrix{
  K_0(A) \ar[r]^{\mathrm{id}-\alpha_*} & K_0(A) \ar[r]^{\iota_*} & K_0(A\rtimes\Z) \ar[d]^{\partial_0} \\
   K_1(A\rtimes\Z) \ar[u]^{\partial_1} & K_1(A) \ar[l]^{\iota_*} & K_1(A) \ar[l]^{\mathrm{id}-\alpha_*} 
}
\]
We provide the proof of the following simple lemma for the reader's convenience.
\begin{lemma}\label{lem:alg} Let $m\leq n$ be positive integers and $\psi$ be the endomorphism on $\frac{\Z}{n\Z}$ given by $\psi(x):=mx$, for $x\in\frac{\Z}{n\Z}$. Then $\ker(\psi)\simeq\frac{\frac{\Z}{n\Z}}{\Ima\psi}\simeq\frac{\Z}{\gcd(m,n)\Z}$.
\end{lemma}
\begin{proof}
    Clearly, $\Ima\psi=\gcd(m,n)\frac{\Z}{n\Z}\simeq\frac{\Z}{\frac{n}{gcd(m,n)}\Z}.$ The result then follows from cardinality arguments.
\end{proof}
\begin{theorem}\label{thm:kt}
Let $p,q\geq 2$ be integers.  For $i=0,1$, we have $$K_i(C^*(\Z[1/pq]\rtimes\Z^2))\simeq\Z^2\oplus\frac{\Z}{\gcd(p-1,q-1)\Z}.$$
\end{theorem}

\begin{proof}
Let $\eta\colon\Z\act\Z[\frac{1}{pq}]$ be given by multiplication by $pq$ and $\g\colon \Z\act\Z[\frac{1}{pq}]\rtimes_\eta\Z$ be given by $\g_n(x,z):=(p^nx,z)$, for $(x,z)\in\Z[\frac{1}{pq}]\rtimes_\eta\Z$. Notice that 
\begin{align*}
    \Z[1/pq]\rtimes\Z^2&\to(\Z[1/pq]\rtimes_\eta\Z)\rtimes_\gamma\Z\\
    (x,m,n)&\mapsto(x,n,m-n)
\end{align*}
is an isomorphism.

Let $A:=C^*(\Z[1/pq]\rtimes_\eta\Z)$. Using the fact that $\Z[1/pq]\rtimes_\eta\Z$ is isomorphic to the \emph{Baumslag-Solitar group} $BS(1,pq)$, it follows from \cite[Theorem 1]{PV18} that $K_0(A)=\Z[1_A]$ (an infinite cyclic group generated by $[1_A]$), and $K_1(A)=\Z\oplus\frac{\Z}{(pq-1)\Z}$, with generators $[u_{(0,1)}]_1$ of infinite order and $[u_{(1,0)}]_1$ of order $pq-1$.

Consider the Pimsner-Voiculescu sequence associated to $\tilde{\gamma}\colon\Z\act A$, where $\tilde{\g}$ is the action induced by $\g$. Surjectivity of $\partial_1$ follows from \cite[Lemma 2]{PV18}. Moreover, $\mathrm{id}-\tilde{\g}_*\colon K_0(A)\to K_0(A)$ is $0$. Therefore, the following sequence is exact:
\[
\xymatrix{
0\ar[r]&\Z\ar[r]^-{\iota_*} & K_0(A\rtimes_{\tilde{\g}}\Z)\ar[r]&\Z\oplus\frac{\Z}{(pq-1)\Z}\\
\ar[r]^-{\mathrm{id}-\tilde{\g}_*}&\Z\oplus\frac{\Z}{(pq-1)\Z}\ar[r]^-{\iota_*}&K_1(A\rtimes_{\tilde{\g}}\Z)\ar[r]&\Z\ar[r]&0.
}
\]
For $(x,y)\in\Z\oplus\frac{\Z}{(1-pq)\Z}$, we have that $(\tilde{\g}_*-\mathrm{id})(x,y)=(0,(p-1)y)$. In particular, by Lemma \ref{lem:alg}, we have that 
$$\frac{\Z\oplus\frac{\Z}{(1-pq)\Z}}{\Ima(\mathrm{id}-\tilde{\g}_*)}\simeq\Z\oplus\frac{\Z}{\gcd(1-pq,p-1)\Z}=\Z\oplus\frac{\Z}{\gcd(p-1,q-1)\Z}.$$ This finishes the computation of $K_1(C^*(\Z[1/pq]\rtimes\Z^2))$. 

By Lemma \ref{lem:alg} again, we also have that $\ker(\mathrm{id}-\tilde{\g}_*)\simeq\Z\oplus\frac{\Z}{\gcd(p-1,q-1)\Z}$. 

Let $\tau\colon A\rtimes_{\tilde{\g}}\Z\cong C^*(\Z[1/pq]\rtimes\Z^2)\to\C$ be the unital $*$-homomorphism associated with the trivial representation of $\Z[1/pq]\rtimes\Z^2$. By identifying $K_0(\C)$ with $\Z$, we obtain that $K_0(\tau)\circ\iota_*=\id_\Z$, and the short exact sequence
\[
\xymatrix{
0\ar[r]&\Z\ar[r]^-{\iota_*} & K_0(A\rtimes_{\tilde{\g}}\Z)\ar[r]& \Z\oplus\frac{\Z}{\gcd(p-1,q-1)\Z}\ar[r]&0
}
\]
left splits. This concludes the proof. 
\end{proof}
\begin{example}
Given integers $p,q\geq 2$, let $\alpha^{p,q}\colon\Z^2\act\Z[\frac{1}{pq}]$ be given by multiplication by $p$ and $q$. For $i=0,1$, we have
\begin{align*}
K_i(C^*(\Z[1/6]\rtimes_{\alpha^{2,3}}\Z^2))&\simeq\Z^2,\\
K_i(C^*(\Z[1/15]\rtimes_{\alpha^{3,5}}\Z^2))&\simeq\Z^2\oplus\Z_2.
\end{align*}
\end{example}

\section{Questions}\label{sec:q}
Let us conclude by posing a few questions that arise naturally in light of our results.

Given multiplicatively independent integers $p,q\geq 2$, we know that the group $\Z[\frac{1}{pq}]\rtimes\Z^2$ is weakly $C^*$-simple by \cite{S-T20}, and that it has the weak unique trace property if and only if Furstenberg's conjecture is true by Corollary~\ref{thm:tf} (this terminology is defined in the discussion following Corollary~\ref{thm:tf}). We have also seen in Example~\ref{ex:Vaes} that $\Z_2\wr\Z$ is weakly $C^*$-simple without having the weak unique trace property. 

 It follows from \cite[Theorem 6.11]{BKKO} that, for groups with countably many subgroups, $C^*$-simplicity is equivalent to the unique trace property. On the other hand, $\Z[\frac{1}{pq}]\rtimes\Z^2$ has countably many subgroups by the argument in \cite[Corollary 8.4]{BLT19} (whereas $\Z_2\wr\Z$ has uncountably many subgroups). 
 \begin{question}Is there an icc group with \emph{countably many subgroups} that is weakly $C^*$-simple but does not have the weak unique trace property? 
 \end{question}  
 Given a faithful, almost minimal action $\G\acts X$ as in Theorem~\ref{thm:pis}, the crossed product $C(X)\rtimes\G$ has the property that any of its \emph{irreducible} representations are either faithful or have finite-dimensional image. 
    \begin{question}
    Can (some of) the theory of just-infinite $C^*$-algebras from \cite{GMR} be extended to the class of $C^*$-algebras for which any irreducible representation is either faithful or has finite-dimensional image? 
    \end{question}
In \cite{Rordam-jinfinite}, it was shown that each infinite dimensional metrizable Choquet simplex arises as the trace simplex of a residually finite-dimensional just-infinite AF-algebra. A $C^*$-algebra is said to be \emph{subhomogeneous} if it is isomorphic to a sub-$C^*$-algebra of $M_n(C_ 0(X))$ for some $n\in\Z_{>0}$ and $X$ locally compact Hausdorff space. A $C^*$-algebra is said to be \emph{approximately subhomogeneous} (ASH-algebra) if it is the inductive limit of a sequence of subhomogeneous $C^*$-algebras. 
\begin{question}
    Let $p,q\geq 2$ be multiplicatively independent integers. Is there a unital ASH-algebra $A$ with the property that any of its \emph{irreducible} representations are either faithful or have finite-dimensional image, and such that $A$ has the same K-theory and primitive ideal space of $C^*(\Z[1/pq]\rtimes\Z^2)$? If yes, can one compute the trace simplex of $A$?
\end{question}
\raggedright
\bibliography{bibliografia}

\end{document}